\documentclass[12pt]{amsart}

\usepackage{mathtools,amsmath,amsthm,amssymb,graphicx,hyperref,comment}

\newcommand{\Z}{\mathbb{Z}}

\newcommand{\divides}{\bigm|}

\newtheorem{theorem}{Theorem}[section]
\newtheorem{lemma}[theorem]{Lemma}

\newtheorem{definition}[theorem]{Definition}
\newtheorem{conjecture}[theorem]{Conjecture}
\newtheorem{remark}[theorem]{Remark}

\newtheorem{example}[theorem]{Example}

\begin{document}

\title{Some properties of even moments of uniform random walks}
\author{Kevin G. Hare}
\thanks{Research of K. G. Hare was supported by NSERC Grant RGPIN-2014-03154}
\address {Dept. of Pure Mathematics,
University of Waterloo,
Waterloo, Ont., N2L 3G1,
Canada}
\email{kghare@uwaterloo.ca}

\author{Ghislain McKay}
\thanks{Research of Ghislain McKay was supported by NSERC Grant RGPIN-2014-03154, 
    the Department of Pure Mathematics, University of Waterloo,
    and CARMA, University of Newcastle}.
\address {Dept. of Pure Mathematics,
University of Waterloo,
Waterloo, Ont., N2L 3G1,
Canada}
\email{glmckay@uwaterloo.ca}

\begin{abstract}
We build upon previous work on the densities of uniform random walks in higher dimensions, exploring some properties of the even moments of these densities and extending a result about their modularity.
\end{abstract}

\maketitle

\section{Introduction}
Consider a short random walk of $n$ steps in $d$ dimensions where each step is 
    of unit length and whose direction is chosen uniformly.  
Following \cite{BSV15}, we let $\nu = \frac{d}{2} - 1$ and denote by 
    $p_n(\nu; x)$ the probability
    density function of the distance $x$ to the origin of this random walk.
This paper will be concerned with the even moments of these 
    random walks.
\begin{definition}
Define 
\[ W_n(\nu; s) = \int_{0}^\infty x^s p_n(\nu; x) \mathrm{dx} \]
as the $s^{th}$ moment of the probability density function.
\end{definition}

We know that

\begin{theorem}[Borwein, Staub, Vignot, Theorem 2.18, \cite{BSV15}]
    For non-negative integers $k$, $W_{n}(\nu;2k)$ is given by
    \begin{equation*}
        \label{moment_def}
        W_{n}(\nu;2k) = \frac{(k+\nu)! \nu!^{n-1}}{(k+n \nu)!} \sum_{k_1 + \cdots + k_n = k} \binom{k}{k_1,\dots,k_n}\binom{k+n\nu}{k_1+\nu,\dots,k_n+\nu}
    \end{equation*}
\end{theorem}

\begin{theorem}[Borwein, Staub, Vignot Example 2.23, \cite{BSV15}]
    For given $\nu$, let $A(\nu)$ be the infinite lower triangular matrix with entries
    \begin{equation*}
        A_{k,j}(\nu) = \binom{k}{j} \frac{(k+\nu)! \nu!}{(k-j+\nu)! (j+\nu)!}
    \end{equation*}
    for row indices $k = 0,1,2,\dots$ and columns entries $j = 0,1,2,\dots$. 
Then the moments $W_{n+1}(\nu; 2k)$ are given by the row sums of $A(\nu)^n$.
\end{theorem}

For a good history of these moments, and random walks in general, see 
\cite{BNSW11, BSV15, BSW13, BSWZ12}.

\begin{example}
For example, the upper corner of $A(0), A(1)$ and $A(2)$ are given below.
\[ A(0) := 
\left[ \begin{array}{ccccccccc} 
1&0&0&0&0&0&0&0&\dots\\
1&1&0&0&0&0&0&0&\\
1&4&1&0&0&0&0&0&\\
1&9&9&1&0&0&0&0&\\
1&16&36&16&1&0&0&0&\\
1&25&100&100&25&1&0&0&\\
1&36&225&400&225&36&1&0&\\
1&49&441&1225&1225&441&49&1&\\
\vdots&&&&&&&&\ddots
\end{array}\right] 
 \]
\[ A(1) := 
\left[ \begin{array}{ccccccccc} 
1&0&0&0&0&0&0&0&\dots\\
1&1&0&0&0&0&0&0&\\
1&3&1&0&0&0&0&0&\\
1&6&6&1&0&0&0&0&\\
1&10&20&10&1&0&0&0&\\
1&15&50&50&15&1&0&0&\\
1&21&105&175&105&21&1&0&\\
1&28&196&490&490&196&28&1&\\
\vdots&&&&&&&&\ddots
\end{array}\right] 
 \]
\[ A(2) := 
\left[ \begin{array}{ccccccccc} 
1&0&0&0&0&0&0&0&\dots\\
1&1&0&0&0&0&0&0&\\
1&8/3&1&0&0&0&0&0&\\
1&5&5&1&0&0&0&0&\\
1&8&15&8&1&0&0&0&\\
1&35/3&35&35&35/3&1&0&0&\\
1&16&70&112&70&16&1&0&\\
1&21&126&294&294&126&21&1&\\
\vdots&&&&&&&&\ddots
\end{array}\right] 
\]

\end{example}

The lower triangular entries of $A(0)$ are the squares of the binomial coefficients $\binom{k}{j}$ and those in $A(1)$ are known as the Naryana numbers \cite[A001263]{Slo15}.
Using these observations about $A(0)$ and $A(1)$, it is easy to observe that 
    all of the coefficients of $A(0)$ and $A(1)$ are integers.
A quick glance at $A(2)$ shows that this is not always true.
It was stated that $A_{k,j}(2) \in \frac{1}{3} \Z$ in \cite{BSV15}.

We define 
    \[ r_\nu := \min\left\{r > 0 : A_{k,j}(\nu) \in \frac{1}{r} \Z, j,k \geq 0 \right\}. \]
Using this notation we see that $r_0 = r_1 =1$ and $r_2 = 3$.
It is not immediately that $r_\nu$ is well defined and finite for all $\nu$, 
    (although we will show that this is the case).

In Section \ref{sec:A Z} we show that 
\begin{theorem}
\label{thm:main}
For $\nu \geq 1$ we have $r_\nu \divides \frac{(2 \nu -1)!}{\nu!}$.
\end{theorem}
This is not best possible.
In Section \ref{sec:backwards} we prove the opposite direction
\begin{theorem}
\label{thm:backwards}
For $\nu \geq 1$ we have $\binom{2 \nu -1}{\nu} \divides r_\nu$.
\end{theorem}
We conjecture that this is in fact best possible.
That is, we conjecture

\begin{conjecture}
\label{conj:main}
For $\nu \geq 1$ we have $r_\nu = \binom{2 \nu -1}{\nu}$.
\end{conjecture}
We present evidence for this conjecture in Section 
    \ref{sec:small} and \ref{sec:evidence}.

Next we consider a result by Borwein, Nuyens, Straub and Wan in \cite{BNSW11} 
    about the modularity of moments.
They showed that 
\begin{theorem}
    For primes $p$, we have
    \begin{equation*}
        W_n(0;2p) \equiv n \mod p.
    \end{equation*}
\end{theorem}

We extend this in Section \ref{sec:mod} to get
\begin{theorem}
\label{thm:mod}
Let 
\begin{itemize}
\item $p = k$ be prime with $2 \nu < p$, or 
\item $p = k + \nu$ be prime with $\nu < p$.
\end{itemize}
Then \begin{equation*}
        W_n(\nu;2k) \equiv n \mod p.
     \end{equation*}
If $p^2 = k$ with $p$ prime then
     \begin{equation*}
        W_n(0;2k) \equiv n \mod p^2.
     \end{equation*}
\end{theorem}
It is worth remarking that if both $p_1 := k$ and $p_2 := k+\nu$ are prime
    with $2 \nu < p_1$ (and hence $\nu < 2 \nu < p_1 < p_2$), then 
    clearly $W_n(\nu; 2k) \equiv n \mod p_1 p_2$ by the 
    Chinese Remainder Theorem.

In Section \ref{sec:conc} we discuss some of the open problems related 
    to this research.

\section{A proof of Theorem \ref{thm:main}: $r_\nu | (2 \nu -1)!/\nu!$}
\label{sec:A Z}
To prove Theorem \ref{thm:main}, we make use of the following remark and lemma:

\begin{remark}
There are multiple equivalent ways of representing $A_{k,j}(\nu)$.
The three most common that we will use are:
    \begin{align*}
        A_{k,j}(\nu)
            &= \binom{k}{j} \frac{(k+\nu)! \nu!}{(k-j+\nu)! (j+\nu)!} & \\
            &= \binom{k}{j} \binom{k+\nu}{j}\binom{j+\nu}{j}^{-1} \\
            &= \binom{k+\nu}{j} \binom{k+\nu}{j+\nu}\binom{k+\nu}{\nu}^{-1}
    \end{align*}
\end{remark}

\begin{lemma}
    \label{lem:gcd}
    For integers $1 \leq \nu \leq j$ we have
    \begin{equation*}
        \gcd((j-\nu+1)(j-\nu+2) \cdots j,\, (j+1)(j+2) \cdots (j+\nu)) \divides (2\nu-1)!
    \end{equation*}
\end{lemma}

\begin{proof}
Let $A_{j,\nu} = j-\nu+1, \dots, j$ and $B_{j, \nu} = j+1, \dots, j+\nu$.
Let $\pi(A_{j,\nu})$ and $\pi(B_{j,\nu})$ be the products of these sequences.
Let $p$ be a prime number and $v_p(x)$ be the $p$-adic valuation of $x$.
We see that for $p^\alpha > 2 \nu$ that there is at most one term in $A_{j,\nu}\cup B_{j,\nu}$
    that is divisible by $p^\alpha$.
Without loss of generality we may assume that such a term, if it exists, is 
    in $A_{j, \nu}$.
We see that $v_p(B_{j,\nu}) = v_p(B_{j+p^\alpha k, \nu})$ for all $k$ by 
    translation.
Further, if there exists a term in $A_{j,\nu}$ that is divisible by $p^\alpha$,
    then, by translations we can assume that this term is divisible by an 
    arbitrarily high power of $p$.
Hence we can assume that, if such a term exists, then we can find a translate
    of this sequence so that
    \[ v_p(\gcd(\pi(A_{j+p^\alpha k, \nu}), \pi(B_{j+p^\alpha k, \nu}))) = 
       v_p(\pi(B_{j+p^\alpha k, \nu})). \]
We see that if $p^\beta \leq \nu$ then there are at most 
    $\left\lceil \frac{\nu}{p^\beta} \right\rceil$ terms in $B_{j+p^\alpha k, \nu}$
    are are divisible by $p^\beta$.
We see that if $\nu < p^\beta \leq 2\nu$ then there are at most 
    $\left\lceil \frac{2\nu}{p^\beta} \right\rceil-1$ terms in 
    $B_{j+p^\alpha k, \nu}$ are are divisible by $p^\beta$.
By Chinese remainder theorem we can find such a $j$ so that both the 
    inequalities are exact.
This gives us that 
\begin{equation}
    v_p(\gcd(\pi(A_{j+p^\alpha k, \nu}), \pi(B_{j+p^\alpha k, \nu}) ))
    \leq \sum_{p^\beta \leq \nu} \left\lceil \frac{\nu}{p^\beta} \right\rceil + 
    \sum_{\nu < p^\beta \leq 2\nu} \left\lceil \frac{2\nu}{p^\beta} \right\rceil -1 
\label{eq:star}
\end{equation}
and moreover there exists a $j$ so that this is exact.

We observe that 
    \[ v_p((2 \nu - 1)!) = \sum_{p^\beta \leq 2 \nu -1} 
    \left\lfloor \frac{2 \nu -1}{p^\beta} \right\rfloor. \]

Observe that if $p^\beta < \nu$ then 
    \[ \left\lfloor \frac{2 \nu -1}{p^\beta} \right\rfloor \geq 
    \left\lceil \frac{\nu}{p^\beta} \right\rceil. \] 
If $p^\beta = \nu$ then 
    \[ \left\lfloor \frac{2 \nu -1}{p^\beta} \right\rfloor =
    \left\lceil \frac{\nu}{p^\beta} \right\rceil  = 1. \] 
If $\nu < p^\beta \leq 2 \nu - 1$ then 
    \[ \left\lfloor \frac{2 \nu -1}{p^\beta} \right\rfloor = 1 \geq
    \left\lceil \frac{2\nu}{p^\beta} \right\rceil -1.  \]
Lastly if $p^\beta = 2 \nu$ then 
    \[ \left\lfloor \frac{2 \nu -1}{p^\beta} \right\rfloor = 0 \geq
    \left\lceil \frac{2\nu}{p^\beta} \right\rceil -1.  \]

Hence $v_p(\gcd(\pi(A_{j,\nu}), \pi(B_{j,\nu})) \leq v_p((2 \nu -1)!)$ which gives that
    \[ \gcd(\pi(A_{j,\nu}), \pi(B_{j,\nu})) \divides (2\nu-1)! \]
as required.
\end{proof}

It is worth remarking that for any fixed $\nu \geq 4$, we can find 
    tighter lower bounds for the $\gcd$ by using \eqref{eq:star} directly.
This can be used to tighten the results of Theorem \ref{thm:main} for 
    specific $\nu$.
Unfortunately even when tightened in this way, we cannot achieve the 
    conjectured bound.
See Table \ref{tab:compare}

\begin{table}
\begin{tabular}{l|l|l}
$\nu$ & Equation \eqref{eq:star} & $(2 \nu -1)!$ \\ \hline
1 &$ 1 $&$ 1$\\ 
2 &$ 2\cdot 3 $&$ 2\cdot 3$\\ 
3 &$ 2^3\cdot 3\cdot 5 $&$ 2^3\cdot 3\cdot 5$\\ 
4 &$ 2^3\cdot 3^2\cdot 5\cdot 7 $&$ 2^4\cdot 3^2\cdot 5\cdot 7$\\ 
5 &$ 2^6\cdot 3^3\cdot 5\cdot 7 $&$ 2^7\cdot 3^4\cdot 5\cdot 7$\\ 
6 &$ 2^6\cdot 3^3\cdot 5^2\cdot 7\cdot 11 $&$ 2^8\cdot 3^4\cdot 5^2\cdot 7\cdot 11$\\ 
7 &$ 2^7\cdot 3^4\cdot 5^2\cdot 7\cdot 11\cdot 13 $&$ 2^{10}\cdot 3^5\cdot 5^2\cdot 7\cdot 11\cdot 13$\\ 
8 &$ 2^7\cdot 3^4\cdot 5^2\cdot 7^2\cdot 11\cdot 13 $&$ 2^{11}\cdot 3^6\cdot 5^3\cdot 7^2\cdot 11\cdot 13$\\ 
9 &$ 2^{11}\cdot 3^4\cdot 5^2\cdot 7^2\cdot 11\cdot 13\cdot 17 $&$ 2^{15}\cdot 3^6\cdot 5^3\cdot 7^2\cdot 11\cdot 13\cdot 17$\\ 
10 &$ 2^{11}\cdot 3^6\cdot 5^2\cdot 7^2\cdot 11\cdot 13\cdot 17\cdot 19 $&$ 2^{16}\cdot 3^8\cdot 5^3\cdot 7^2\cdot 11\cdot 13\cdot 17\cdot 19$\\ 
\end{tabular}
\caption{Prime factorization of Eq \eqref{eq:star} and $(2\nu -1)!$.}
\label{tab:compare}
\end{table}

We are now ready to prove Theorem \ref{thm:main}.

\begin{proof}[Proof of Theorem \ref{thm:main}]
    We fix integers $\nu \geq 0$ and $0 \leq j \leq k$. We consider 2 cases:

    If $0 \leq j \leq \nu-1$ then we have
    \begin{equation*}
        A_{k,j}(\nu) = \binom{k}{j} \binom{k+\nu}{j} \binom{j+\nu}{j}^{-1}
                     = \binom{k}{j}\binom{k+\nu}{j} j! \frac{\nu!}{(j+\nu)!}
    \end{equation*}
    by our assumption on $j$ we know $j+\nu \leq 2\nu-1$, hence $(j+\nu)! \divides (2\nu-1)!$, and therefore
    \begin{equation*}
        A_{k,j}(\nu) \in \tfrac{\nu!}{(j+\nu)!} \Z \subseteq \tfrac{\nu!}{(2\nu-1)!} \Z.
    \end{equation*}

    Otherwise we may assume that $j \geq \nu$.  Then we have
    \begin{align*}
        A_{k,j}(\nu)
            &= \binom{k+\nu}{j}\binom{k+\nu}{j+\nu} \binom{k+\nu}{\nu}^{-1} & \\
            &= \frac{(k+\nu) \cdots (k+1) \cdot k \cdots (k+\nu-j+1)}{j!} \cdot \\
            &  \frac{(k+\nu) \cdots (k-j+1)}{(j+\nu)!} \cdot \frac{\nu!}{(k+\nu) \cdots (k+1)} & \\
            &= \frac{k \cdots (k+\nu-j+1)}{j!} \cdot \frac{(k+\nu) \cdots (k-j+1)}{(j+\nu)!} \cdot \nu! & \\
            &= \frac{(k+\nu) \cdots (k+1)}{(j+\nu) \cdots (j+1) \cdot j \cdots (j-\nu+1)} \binom{k}{j-\nu} \binom{k}{j} \nu!
    \end{align*}

    Next observe that
    \begin{align*}
        \frac{(k+\nu) \cdots (k+1)}{(j+\nu) \cdots (j+1)} \binom{k}{j-\nu} \binom{k}{j} &= \binom{k}{j-\nu} \binom{k+\nu}{j+\nu} & \\
        \frac{(k+\nu) \cdots (k+1)}{(j) \cdots (j-\nu+1)} \binom{k}{j-\nu} \binom{k}{j} &= \binom{k+\nu}{j} \binom{k}{j} &
    \end{align*}
    are both integers, hence there exists $p,q \in \Z$ such that
    \begin{equation*}
        A_{k,j}(\nu)
            = \frac{(k+\nu) \cdots (k+1)}{(j+\nu) \cdots (j+1) \cdot j \cdots (j-\nu+1)} \binom{k}{j-\nu} \binom{k}{j} \nu!
            = \frac{p}{q} \nu!
    \end{equation*}
    and where $q \divides \gcd((j+\nu) \cdots (j+1),\,j \cdots (j-\nu+1))$.

    By Lemma \ref{lem:gcd} and transitivity of divisibility, $q \divides (2\nu-1)!$ hence there exists $p'$ such that
    \begin{equation*}
        A_{k,j}(\nu) = p' \cdot \frac{\nu!}{(2\nu-1)!}.
    \end{equation*}
    Thus, for all integers $\nu \geq 0$ we have $r_\nu \divides \tfrac{\nu!}{(2\nu-1)!}$ as desired.
\end{proof}

\section{A proof of Theorem \ref{thm:backwards}: $\binom{2\nu-1}{\nu} \divides r_\nu$}
\label{sec:backwards}

Theorem \ref{thm:backwards} is an immediate corollary of:
\begin{lemma}
Let $p^\alpha | \binom{2 \nu -1 }{\nu}$.
Let $p^r \geq p^\alpha$ and $p^r > \nu$.
Then the denominator of 
$A_{p^{r}-1, \nu-1}(\nu)$ is divisible by $p^\alpha$.
\end{lemma}

\begin{proof}
Let $p^\alpha | \binom{2 \nu -1 }{\nu}$.
Let $p^r \geq p^\alpha$ and $p^r > \nu$.
Notice that 
\[
A_{p^{r}-1, \nu-1}(\nu)
= \binom{p^r+\nu-1}{\nu-1} \binom{p^r-1}{\nu-1} \binom{2 \nu -1}{\nu-1}^{-1}.
\]
Consider the first term.
\[
\binom{p^r + \nu -1}{\nu-1} = \frac{(p^r+\nu-1) \cdots (p^r+1)}{(\nu-1) \cdots 1}.
\]
Observe that each factor of the top is equivalent mod $p^r$ to the matching
    factor in the bottom.
Hence $\binom{p^r + \nu -1}{\nu-1} \equiv 1 \mod p$.

The second term is similar, with each term on the top equivalent mod $p^r$ to
   the additive inverse of the associated factor on the bottom.
Hence $\binom{p^r -1}{\nu-1} \equiv (-1)^\nu \mod p$.

Hence 
\[
A_{p^{r}-1, \nu-1}(\nu) = \frac{1}{p^\alpha} \cdot \frac{a}{b}
\]
with $p$ co-prime to $a$.
\end{proof}

\section{The case $\nu = 3$ and $\nu = 4$}
\label{sec:small}
We see that $r_1 = 1 = \binom{1}{1}$ and $r_2 = 3 = \binom{3}{2}$.
In this section we show the next two cases of Conjecture \ref{conj:main} 
    hold, namely that 
            $r_3 = 10 = \binom{5}{3}$ and $r_4 = 35 = \binom{7}{4}$.

We first need the Lemma
\begin{lemma}
    \label{lem:kummer}
    Let $n$ and $k$ be non-negative integers. If $n$ is even and $k$ is odd then $\binom{n}{k}$ is even.
\end{lemma}

\begin{proof}
    By Kummer's theorem \cite{Kum1852}, 2 divides $\binom{n}{k}$ when there is at least one carry when $k$ and $n-k$ are added in base 2.
    Since $n$ is even and $k$ is odd, $n-k$ is odd.
    The least significant bit of an odd integer represented in base 2 is always 1. Hence both $k$ and $n-k$ have a 1 in the least significant place. Thus when they are added, this will result in a carry. So 2 divides $\binom{n}{k}$.
\end{proof}

We now follow the proof of Theorem \ref{thm:main} using $\nu = 3$ to show:
\begin{theorem}
\label{thm:nu=3}
    Conjecture \ref{conj:main} holds for $\nu = 3$.
    That is $r_{3} = \binom{5}{3} = 10$.
\end{theorem}

\begin{proof}
    We have that $10 | r_3$ by Theorem \ref{thm:backwards}.

    As in the proof of Theorem \ref{thm:main}, we first consider the case where $0 \leq j \leq 2$.
    A quick calculation shows that 
    \begin{align*}
        A_{k,0}(3)\binom{5}{3} & = 10 \\ 
        A_{k,1}(3)\binom{5}{3} & = \frac{5 (k+3) k}{2} \\ 
        A_{k,2}(3)\binom{5}{3} & = \frac{(k-1) (k+2) (k+3) k}{4}
    \end{align*}
    By considering the cases of $k$ even or odd, we see that all of these values are always integers, and hence 
    $A_{k,0}(3), A_{k,1}(3), A_{k,2}(3) \in \frac{1}{10}\mathbb{Z}$.

    If $j \geq 3$ then, as in the proof of Theorem \ref{thm:main}, we have
    \begin{align*}
A_{k,j}(3) & = \frac{3!}{(j+3)(j+2)(j+1)}\binom{k+3}{j}\binom{k}{j}\\
           &  = \frac{3!}{j(j-1)(j-2)}\binom{k}{j-3}\binom{k+3}{j+3}.
    \end{align*}
We see that if $8 \nmid \gcd((j+3)(j+2)(j+1),\, j (j-1) (j-2))$ then 
    \[ A_{k,j}(3) \in \frac{2! 3!}{5!} \Z \] as required.
Hence we may assume that $8 \divides \gcd((j+3)(j+2)(j+1),\, j (j-1) (j-2))$.
If $j$ is even then $8 \divides(j+3)(j+2)(j+1)$ implies that 
    $j \equiv 6 \mod 8$.
We observe that $8 \divides j (j-1)(j-2)$ and $16 \nmid j(j-1)(j-2)$.
In this case we observe that one of $\binom{k}{j-3}$ and $\binom{k+3}{j+3}$ is 
    also even by Lemma \ref{lem:kummer}.
Hence we may write \[ A_{k,j}(3) = 
    \frac{2}{8} \cdot \frac{p}{q}\] where $q$ is odd.
This implies that
    \[ A_{k,j}(3) \in \frac{2! 3!}{5!} \Z \] as required.

Similarly if $j$ is odd, then $j \equiv 1 \mod 8$,
    and $8 \divides (j+1)(j+2)(j+3)$ and $16 \nmid (j+1)(j+2)(j+3)$.
Further one of 
    $\binom{k+3}{j}$ and $\binom{k}{j}$ is even, and hence 
\[ A_{k,j}(3) = 
    \frac{2}{8} \cdot \frac{p}{q}\] where $q$ is odd.
Again this implies that
    \[ A_{k,j}(3) \in \frac{2! 3!}{5!} \Z \] as required.
\end{proof}

\begin{theorem}
\label{thm:nu=4}
    Conjecture \ref{conj:main} holds for $\nu = 4$.
    That is $r_{4} = \binom{7}{4} = 35$.
\end{theorem}

\begin{proof}
    We have that $35 | r_4$ by Theorem \ref{thm:backwards}.

    As in the proof of the previous theorem, we first consider the case where $0 \leq j \leq 3$.
    A quick calculation shows that 
    \begin{align*}
        A_{k,0}(4)\binom{7}{4} & = 35 \\ 
        A_{k,1}(4)\binom{7}{4} & = 7 k (k+4) \\
        A_{k,2}(4)\binom{7}{4} & = \frac{7 (k-1) k (k+3) (k+4)}{12} \\
        A_{k,3}(4)\binom{7}{4} & = \frac{(k-2) (k-1) k (k+2) (k+3) (k+4)}{36}
    \end{align*}
    By considering the various cases for $k$ mod $12$ (resp. $36$), we see that
    these expressions are always integers, and hence
    $A_{k,0}(4), A_{k,1}(4), A_{k,2}(4), A_{k,3}(4) \in \frac{1}{35}\Z$.

If $j \geq 4$ then, as in the previous proof, we have
    \begin{align*}
A_{k,j}(4) & = \frac{4!}{(j+4)(j+3)(j+2)(j+1)}\binom{k+4}{j}\binom{k}{j}  \\
           & = \frac{4!}{j(j-1)(j-2)(j-3)}\binom{k}{j-4}\binom{k+4}{j+4}.
    \end{align*}

From equation \eqref{eq:star} or Table \ref{tab:compare} we have that 
    \[ \gcd((j+4)(j+3)(j+2)(j+1), j (j-1)(j-2)(j-3)) 
       \divides 7!/2 \]
Hence we have that $A_{k,j}(4) \in \frac{2 \cdot 4!}{7!} \Z$.
We still need to show that there is an additional factor of $3$ 
    in the numerator.

To prove the result, we need to show that one of three things
    occurs
\begin{itemize}
\item $9 \nmid \gcd((j+4)(j+3)(j+2)(j+1), j (j-1)(j-2)(j-3))$
\item $3 \mid \binom{k+4}{j}\binom{k}{j}$, or  
\item $3 \mid \binom{k}{j-4}\binom{k+4}{j+4}$.
\end{itemize}
If $(j+4)(j+3)(j+2)(j+1) \equiv j(j-1)(j-2)(j-3) \equiv 0 \mod 9$ then 
    $j \equiv 2 \mod 9$ or $j \equiv 6 \mod 9$. 
Hence if $j \equiv 0,1,3,4,5,7,8 \mod 9$ then 
    $A_{k,j}(4) \in \frac{3! \cdot 4!}{7!} \Z$ as required.

If $j \equiv 2 \mod 9$.
    then $27 \nmid (j+1) (j+2) (j+3) (j+4)$ so we have that $9$
    divides the gcd exactly.

Consider
\begin{align}
\binom{k+4}{j}\binom{k}{j}
      & = \frac {f_{a,b}(k,j)} {g_{a,b}(k,j)} \binom{k+a}{j} \binom{k+b}{j}
\end{align}
where $f_{a,b}(k,j)$ and $g_{a,b}(k,j)$ are polynomials.
With careful choices of $a$ and $b$ we can construct $f_{a,b}$ and $g_{a,b}$ 
    such that $f_{a,b}(k,j)$ will have more factors of $3$ than $g_{a,b}$.

For example, if $a = b = 2$ then 
\begin{align*}
f_{2,2}(k,j) & = (k+4) (k+3) (k+2-j) (k-j+1) \\
g_{2,2}(k,j) & = (k-j+4) (k-j+3) (k+2) (k+1)
\end{align*}
Using the fact that $j \equiv 2 \mod 3$, we see that for 
    $k \equiv 0 \mod 3$ that $f_{2,2}(k,j) \equiv 0 \mod 3$ and 
    $g_{2,2}(k,j) \equiv 1 \mod 3$ and hence 
    $\binom{k+4}{j} \binom{k}{j} \equiv 0 \mod 3$.
A similar argument is given for $k \equiv 1 \mod 3$ and $k \equiv 2 \mod 3$,
    summarized in Table \ref{tab:proof1}.
Hence if $j \equiv 2 \mod 9$ then 
    $A_{k,j}(4) \in \frac{3! \cdot 4!}{7!} \Z$ as required.

\begin{table}
\begin{tabular}{l|l|l|l|l|l}
$k$ & $j$ & $a$ & $b$ & & \\ \hline 
$k \equiv 0 \mod 3$ & $j \equiv 2 \mod 3$ & 2 & 2 & $f \equiv 0 \mod 3$ & $g \equiv 1 \mod 3$ \\
$k \equiv 1 \mod 3$ & $j \equiv 2 \mod 3$ & 4 & 1 & $f \equiv 0 \mod 3$ & $g \equiv 2 \mod 3$ \\
$k \equiv 2 \mod 3$ & $j \equiv 2 \mod 3$ & 0 & 3 & $f \equiv 0 \mod 3$ & $g \equiv 2 \mod 3$
\end{tabular}
\caption{Cases when $j \equiv 2 \mod 9$}
\label{tab:proof1}
\end{table}

If $j \equiv 6 \mod 9$
    then $27 \nmid j (j-1) (j-2) (j-3)$ so we have that $9$
    divides the gcd exactly.

Consider
\begin{align}
\binom{k+4}{j+4}\binom{k}{j-4}
      & = \frac {f_{a,b}(k,j)} {g_{a,b}(k,j)} \binom{k+a}{j-4} \binom{k+b}{j+4}
\end{align}

As before, we can break this into cases, as described in Table \ref{tab:proof2}
\begin{table}
\begin{tabular}{l|l|l|l|l|l}
$k$ & $j$ & $a$ & $b$ & & \\ \hline 
$k \equiv 0 \mod 3$ & $j \equiv 0 \mod 3$ & 2 & 4 & $f \equiv 0 \mod 3$ & $g \equiv 2 \mod 3$ \\
$k \equiv 1 \mod 3$ & $j \equiv 0 \mod 3$ & 1 & 3 & $f \equiv 0 \mod 3$ & $g \equiv 1 \mod 3$ \\
$k \equiv 2 \mod 3$ & $j \equiv 0 \mod 3$ & 0 & 2 & $f \equiv 0 \mod 3$ & $g \equiv 2 \mod 3$ \\
\end{tabular}
\caption{Cases when $j \equiv 6 \mod 9$}
\label{tab:proof2}
\end{table}
\end{proof}

\section{Additional support for Conjecture \ref{conj:main}}
\label{sec:evidence}

We have computationally checked that 
    for all $k, j, \nu \leq 200$ that Conjecture \ref{conj:main} holds,
Further, using the techniques of Theorems \ref{thm:nu=3} and \ref{thm:nu=4}
    we have computationally verified that for all $j, \nu \leq 15$ 
    and all $k$ that Conjecture \ref{conj:main} holds.
It is not unreasonable to think that Conjecture \ref{conj:main} can hold in general. 
Indeed, if we plot the non-integer entries in the lower triangular part of $A(\nu)$ and colour them based on the prime factorization of their denominators in reduced form we obtain the fractal pattern seen in Figure (\ref{frac_fig}).
This suggests that there is far more structure to the matrix $A(\nu)$ that
    we are currently exploiting.
We note that from equation $\eqref{eq:star}$ combined with Theorem 
    \ref{thm:main} we would be able to prove that 
    $r_5 | 2^3 \cdot 3^2 \cdot 7$.
We conjecture that $r_5 = \binom{7}{4} = 2 \cdot 3^2 \cdot 7$.
In this image of $A(5)$, denominators are coloured red for 2, blue for 3, green for 7 and orange for $3^2$.
If the denominator had contained any additional factors of $2$, $3$ or $5$
    then we would have coloured this value black. 
None occurred.
Assuming that primes always give rise to the associated fractals early on, 
    as seen in Figure \ref{frac_fig}, we would be led to believe that 
    $4 \nmid r_5$.

\begin{figure}
    \centering
    \includegraphics[scale=0.23]{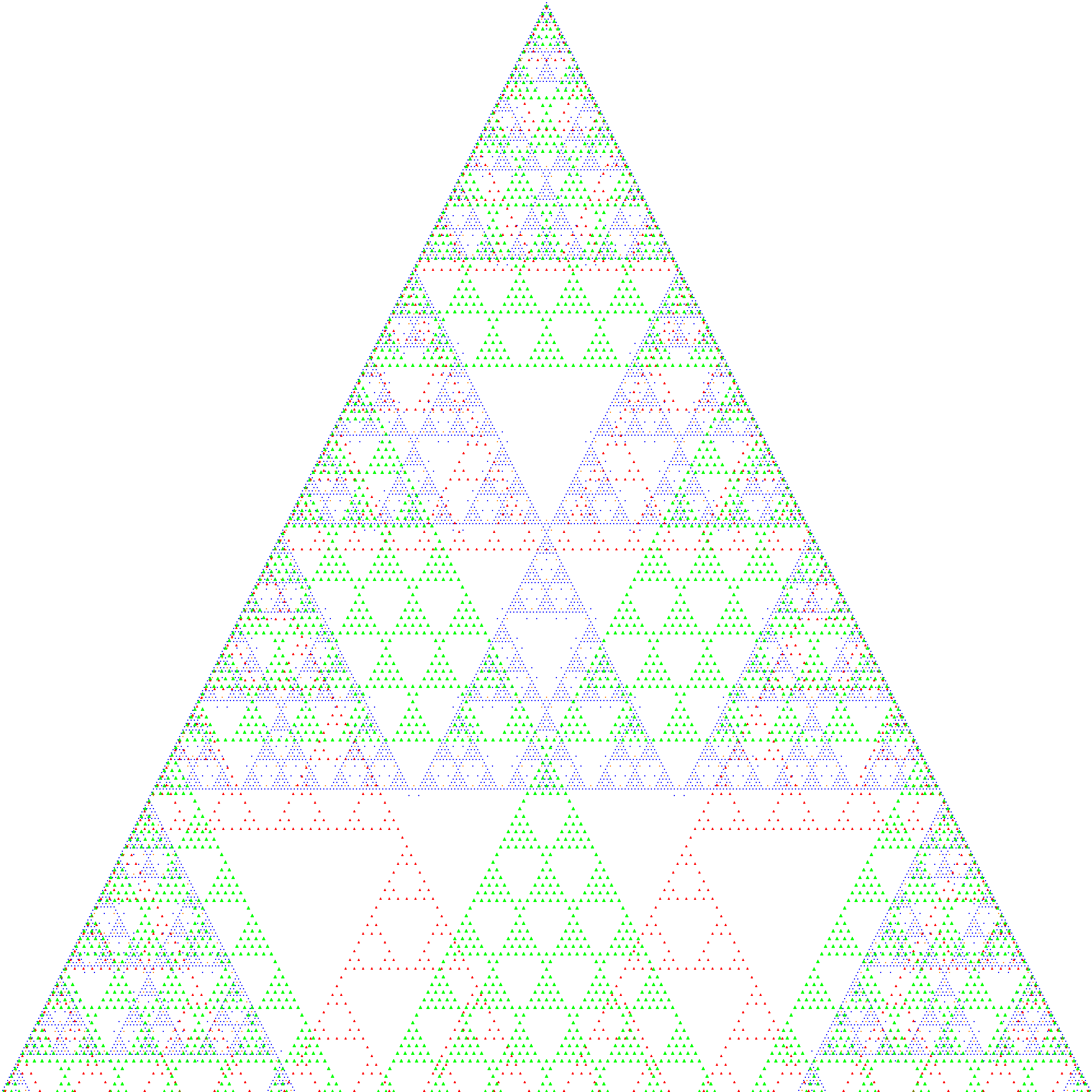}
    \caption{Non-integer entries of the first 1000 rows of $A(5)$}
    \label{frac_fig}
\end{figure}

\section{Proof of Theorem \ref{thm:mod}: $W_n(v;2k) \equiv n$}
\label{sec:mod}

\begin{proof}[Proof of Theorem \ref{thm:mod}]
We rewrite (\ref{moment_def}) as
    \begin{equation*}
        W_{n}(\nu;2k) = \sum_{k_1 + \cdots + k_n = k} \frac{k! \cdot (k+\nu)! \cdot \nu!^{n-1}}
                             {k_1! \cdots k_n! \cdot (k_1+\nu)! \cdots (k_n+\nu)!}
    \end{equation*}
Let $p = k$ be prime with $2 \nu < p$
    or let $p = k + \nu$ be prime with $\nu < p$.
We claim that there does not exist indices $1 \leq i < j \leq n$ such that $k_i+\nu \geq p$ and $k_j+\nu \geq p$. 
Indeed, this would lead to
    \begin{equation*}
        2p \leq k_i + k_j + 2\nu \leq (k_1+\cdots+k_n) + 2\nu = k + 2\nu.
    \end{equation*}
If $p = k$ then $2 \nu < p$ by assumption and hence $2p \leq k + 2 \nu < 2p$, a contradiction.
If $p = k+\nu$ then $\nu < p$ by assumption and hence $2p \leq (k+\nu) + \nu < 2p$, a contradiction.

If instead $k = p^2$ and $\nu = 0$ it is easy to see that 
    there does not exist indices $1 \leq i < j \leq n$ such that $k_i+\nu \geq p^2$ and $k_j+\nu \geq p^2$. 

We consider 2 cases:

    If there exists $1 \leq i \leq n$ such that $k_i = k$ then clearly $k_j = 0$ for $j \neq i$ and hence
    \begin{equation*}
         \frac{k! \cdot (k+\nu)! \cdot \nu!^{n-1}}{k_1! \cdots k_n \cdot (k_1+\nu)! \cdots (k_n+\nu)!}
             = \frac{k! \cdot (k+\nu)! \cdot \nu!^{n-1}}{k! \cdot 0! \cdots 0! \cdot (k+\nu)! \cdot \nu! \cdots \nu!} = 1
    \end{equation*}

Assume that $k_i < k$ for all $1 \leq i \leq n$.

If $p = k$ we see that $p | k!$ and $p | (k+\nu)!$. 
We further see that at most one term in the denominator is divisible by $p$.
Hence 
    \begin{equation*}
        \frac{k! \cdot (k+\nu)! \cdot \nu!^{n-1}}{k_1! \cdots k_n! \cdot (k_1+\nu)! \cdots (k_n+\nu)!}
    \end{equation*}
can be written as $p \frac{a}{b}$ where $p \nmid b$, and thus is equivalent to $0$ mod $p$.

If $p = k+\nu$ we see that $p | (k+\nu)!$.
We further see that no term in the denominator is divisible by $p$.
Hence 
    \begin{equation*}
        \frac{k! \cdot (k+\nu)! \cdot \nu!^{n-1}}{k_1! \cdots k_n! \cdot (k_1+\nu)! \cdots (k_n+\nu)!}
    \end{equation*}
can be written as $p \frac{a}{b}$ where $p \nmid b$, and thus is equivalent to $0$ mod $p$.

If $p^2 = k$ and $\nu = 0$ we see that $p^{p+1} | k!$ and 
    $p^{p+1} | (k+\nu)!$.
We further see that we have at most $2 p$ factors of $p$ in the 
    denominator, with equality only if $p | k_i$ for all $i$.
Hence 
    \begin{equation*}
        \frac{k! \cdot (k+\nu)! \cdot \nu!^{n-1}}{k_1! \cdots k_n! \cdot (k_1+\nu)! \cdots (k_n+\nu)!}
    \end{equation*}
can be written as $p^2 \frac{a}{b}$ where $p \nmid b$, and thus is equivalent to $0$ mod $p^2$.

    Thus there are only $n$ terms in the sum for $W_{n}(\nu;2k)$ which are not $0 \mod p$ (resp $0 \mod p^2$), namely when $k_i = k$ for some $k$.  
In this case the term is $1 \mod p$ (resp $1 \mod p^2$) hence
    \begin{equation*}
        W_{n}(\nu;2k) \equiv n \mod p \ \ \ \ 
        (\mathrm{resp.}\ W_{n}(0;2k) \equiv n \mod p^2)
    \end{equation*}
\end{proof}

\section{Comments}
\label{sec:conc}

We showed in Section \ref{sec:small} that Conjecture \ref{conj:main} held
    for the case $\nu = 3$ and $\nu = 4$.
It is probably that this technique could be extended computationally for 
    any fixed $\nu$, although this is not clear.
It is not clear that this technique would be extendable to arbitrary $\nu$ 
    without additional ideas.

In Section \ref{sec:mod} we showed how the ideas of modularity of 
    $W_n(\nu; k)$ could be extended to $k = p^2$ or $\nu > 0$.
It appears that something is also happening in the case when
    $k = p^2 \neq 4$ and $\nu = 1$, although it is unclear how 
    one would prove this.
There are most likely many other relations that can be found when considering
    $W_n$ modulo a well chosen prime power.

\section{Acknowledgements}

The authors would like to thank Jon Borwein for many useful discussions and 
    suggestions, without which this paper would not have been possible.

\end{document}